\documentclass[12pt]{amsart}  
      
\date{October 12, 2025}
   
\makeatletter
\@namedef{subjclassname@2020}{\textup{2020} Mathematics Subject Classification}
\makeatother

\newcommand{\adj}{{precisely }} 

\usepackage{amsmath}
\usepackage{amssymb}
\usepackage{mathtools}
\numberwithin{equation}{section} 
\usepackage{graphicx} 
\usepackage{overpic}
\usepackage{enumerate}
\usepackage{verbatim} 
\usepackage{wrapfig}
\usepackage[font=small,labelfont=rm]{caption}
\usepackage{hyperref}

\newcommand{\la}{\langle}	   	
\newcommand{\ra}{\rangle}          	

\newcommand{\bi}[1]{{\boldmath\textbf{\itshape{#1}}}} 

\newcommand{\R}{{\mathbb R}}		
\newcommand{\e}{\varepsilon}            
\newcommand{\cn}{\colon}                

\newcommand{\dist}{\mathop{\rm dist}}

\newcommand{\Z}{{\mathbb Z}}

\newtheorem{thm}{Theorem}[section]
\newtheorem{lemma}[thm]{Lemma}
\newtheorem*{introthm}{Main Theorem}
\newtheorem{prop}[thm]{Proposition}

\newtheorem{conjecture}[thm]{Conjecture}

\theoremstyle{definition}
\newtheorem{defn}[thm]{Definition}

\theoremstyle{remark}

\newcommand{\f}{\partial}

\author{Maria Girardi}
\address{ 
  Department of Mathematics, University of South Carolina, Columbia, SC 29208}
\email{girardi@math.sc.edu}

\author{Ralph Howard}
\address{  
Department of Mathematics, University of South Carolina, Columbia, SC 29208}
\email{howard@math.sc.edu}

\title[Nowhere differentiable multivariate functions]{Continuous nowhere 
  differentiable multivariate functions}

\usepackage[dvipsnames]{xcolor}
\definecolor{FuchsiaCode}{RGB}{255,0,255}
\definecolor{PurplePME}{RGB}{140,0,255}

\catcode`@=11 \@mparswitchfalse  
\newcounter{mnotecount}  
\renewcommand{\themnotecount}{\arabic{mnotecount}} 
\newcommand{\mnote}[1]
{\protect{\stepcounter{mnotecount}}$^{\mbox{\footnotesize  $
      \bullet$\themnotecount}}$\marginpar{\parbox[b]{1.2in}{\raggedright\tiny\em
	 \themnotecount:\! #1}} }
\newcommand{\fnb}[3]
{\protect{\stepcounter{mnotecount}}  
   \marginpar{\mbox{{\Small $\diamond$\themnotecount\,$\downarrow$}~{\tiny #1}}} 
   \hskip -12 pt \footnote[\value{mnotecount}]{#2 #3}\addtocounter{footnote}{0}}

\begin{document}

\begin{abstract}
Let $U$ be an open set in $\R^d$.  A
continuous function $f\cn U \to \R $
is \bi{strongly nowhere differentiable}
 if and only if 
for each   $\gamma\in(0,1]$ and
for each unit speed $C^{1,\gamma}$ curve 
$c\cn [a,b] \to U$, 
the composition $f\circ c \cn [a,b] \to \R$
is nowhere differentiable on  $(a,b)$.
For bounded $U$, let
$\overline U$ be the closure of $U$ and
$C(\overline U)$ be the Banach space of continuous
real-valued functions on $\overline U$
with the sup 
norm.
\textbf{Theorem.}  \emph{In the sense of the  Baire category 
theorem, almost every $f\in C(\overline U)$ is strongly 
nowhere differentiable on $U$.}
\end{abstract}

\maketitle

\section{Introduction}
\label{sec:intro}

The existence of continuous nowhere differentiable functions of a single  
variable has been known since the 19th century.  The earliest proposed example
is due to Riemann who claimed the function
$R(x) = \sum_{n=1}^\infty n^{-2}\sin(\pi n^2 x)$
was nowhere differentiable. In 1916,
G.~ H.\ Hardy proved the  function $R$ was non-differentiable at
all irrational points and all rational points of the form $(2p+1)/(2q)$
or $2p/(4q+1)$ with $p$ and~$q$ integers~\cite{Hardy:Reim_fcn}.  In 1970,
Gerver \cite{Gerver70} proved $R'(x) = -\pi/2$ for rational points of the form
$x = (2p+1)/(2q+1)$.   So Riemann's example is  only almost everywhere
non-differentiable.    In 1872,  Weierstrass introduced his famous example
$W(x) =\sum_{n=1}^\infty a^n \cos(b^n \pi x) $ and proved  the
function $W$ is nowhere differentiable when  $b$ is an odd integer, 
$a\in (0,1)$, and $ab>1 + 3\pi/2$.

By the 1930s it was realized that not only do continuous nowhere
differentiable functions exist, but also they are in some sense typical:
Banach \cite{Banach31} and  Mazurkiewicz \cite{Mazurkiewicz}
showed that in $C([0,1])$,
the metric space of continuous real valued functions
with the supremum norm, the set of nowhere differentiable functions
contains a dense $G_\delta$ set.
Therefore, in the sense of the Baire category theorem,
almost every continuous function is nowhere differentiable.

For a nice discussion of Baire's theorem and
its application to ideas related to nowhere
differentiable functions see the beautiful
monograph of Oxtoby \cite{Oxtoby:book}.
Jarnicki and Plug's  comprehensive book
\cite{Jarnicki_Pflug} on nowhere differentiable functions
has ample constructions along with  history.

In this note we consider analogous questions for multivariate functions.
Let $U \subseteq \R^d $ be open and $f\cn U \to \R$.  The standard
definition of $f$ being differentiable at $x_0\in U$ is that there is
a linear map $f'(x_0)\cn \R^d \to \R$ such that 
$$
\| f(x) - f(x_0)- f'(x_0)(x-x_0)\| = \|x-x_0\|\epsilon(x;x_0)
$$
where $\lim_{x\to x_0}\epsilon(x;x_0)=0$.
Let
$g_j\cn \R \to \R$ be a continuous
nowhere differentiable function of one variable.
Then the function
$$
f(x_1,x_2 ,\ldots, x_d) = g_1(x_1) + g_2(x_2) + \cdots + g_d(x_d)
$$
is nowhere differentiable; yet, 
its restriction to a lower dimensional submanifold can be everywhere  differentiable.
For example if  $d=2$ and  $g_2=-g_1$  then  $f(x,y) = g_1(x)-g_1(y)$
is nowhere differentiable on the plane,
but $f(x,x)\equiv 0$ and so $f$ restricted to
the line $y=x$ is $C^\infty$.

Recall $g:I\to\R^d$,
where $I$ is an interval in $\R$,
is $\alpha$-H\"older continuous (in short, is $C^{0,\alpha}$)
where $\alpha>0$
if and only if
there is a constant $\rho$ so that
$\| g(x_2) - g(x_1)\| \leq \rho |x_2 -x_1|^\alpha$ for all $x_1, x_2 \in I$.
Also, $g\in C^{1,\alpha}$
if and only if
its first derivative exists and 
$g'\in C^{0,\alpha}$.

\begin{defn}\label{def:test_curve}
A \bi{test curve} in $\R^d$ is a $C^1$ function $c\cn [a,b] \to \R^d$,
defined on an  interval $[a,b]$,
satisfying
$c$ has unit speed 
(i.e.,  $\|c'(s)\| = 1$ for each $s\in[a,b]$)
and $c$  is $C^{1,\gamma}$ for some $\gamma\in(0,1]$
(i.e., there is a constant~$\rho$ so that 
  $\|c'(s_2) - c'(s_1)\|\le \rho |s_2-s_1|^\gamma$
  for all $s_1, s_2\in[a,b]$).
\end{defn}

\begin{defn}
Let $f\cn U \to \R$ where $U$ is an open set in $\R^d$.  Then
$f$ is \bi{strongly nowhere differentiable} if and only if
for every test curve $c\cn [a,b] \to U$ the composition
$f\circ c$ 
is nowhere differentiable on~$(a,b)$.
\end{defn}

Let $f$ be strongly nowhere differentiable on $U$ and $x\in U$.
Then $f$ has no directional derivative at $x$.
Indeed,  let   $v\in  \R^d$ with $\| v\| = 1$.
The line segment parameterized by
$c(s)= x+sv$ for $s\in  [-\delta, \delta]$ with $\delta$
small is  a $U$-valued test curve 
so $f\circ c$ has no derivative at~$0$.

For a bounded 
set $U$ in $\R^d$ with closure $\overline U$, 
let $C(\overline U)$ be the
Banach space of all continuous functions
$f\cn \overline U\to \R$ 
with the supremum norm.

\begin{introthm}
  Let $U$ be a bounded open set in $\R^d$ with closure $\overline U$
  and let $C(\overline{U})$ be the Banach space of
  continuous functions $f\cn \overline U\to \R$   
  endowed with the supremum norm.
    The set of functions in $C(\overline U)$
    that are strongly
    nowhere differentiable on $U$  contains a 
    $G_\delta$ subset that is dense in  $C(\overline U)$.
\end{introthm}

More informally, in the sense of the Baire Category Theorem, for
almost every $f\in C(\overline U)$ the restriction $f\big|_U$ is strongly
nowhere differentiable.
The Main Theorem is proved in Section \ref{section:PfMainThm}.

\section{Construction of  auxiliary functions}     
\label{section:main_example}

\begin{defn}
  A function $f\cn \R \to \R$ is \bi{\adj $C^{0,\alpha}$} where
  $\alpha>0$
  if and only if 
  there is a constant $C>0$ so that 
  $|f(x)-f(y)| \le C|x-y|^\alpha$ for all $x,y\in \R$,
  and for all $\beta>\alpha$ and $y\in \R$
    $$
    \limsup_{x\to y} \frac{|f(x)- f(y)|}{|x-y|^\beta} = \infty.
    $$
\end{defn}

\begin{defn}
  Let $I$ be an interval of $\R$ and $x_0\in I$. 
  A function  $f \cn I \to \R$
  is \bi{\adj $C^{0,\alpha}$ at $x_0$} where $\alpha>0$
  if and only if 
        $$
        \limsup_{x\to x_0} \frac{|f(x) -f(x_0)|}{|x-x_0|^\alpha} < \infty
        $$
        and for any $\beta > \alpha$  
        \begin{equation} \label{eq:b-infty} 
        \limsup_{x\to x_0}  \frac{|f(x) -f(x_0)|}{|x-x_0|^\beta} = \infty.
    \end{equation} 
\end{defn}

\begin{thm}\label{thm:strict_fcn}
For each $\alpha$ with $0<\alpha< 1$ there
is a \adj $C^{0,\alpha}$ function $f\cn  \R \to \R$,
which is also  even,  $2$-periodic,  and thus   bounded.  
\end{thm}

\begin{proof}
See the appendix.
\end{proof}

\begin{prop}\label{prop:not_diff}
    If $f\cn [a,b]\to \R$ is \adj  $C^{0,\alpha}$ at $x_0\in [a,b]$,
    with $0<\alpha<1$, then $f$ is not differentiable at $x_0$.
\end{prop}

\begin{proof}
This is a direct consequence of the definition of  \adj $C^{0,\alpha}$
at $x_0$.
As $\alpha<1$, taking $\beta= 1$ gives
$$
\limsup_{x\to x_0} \left| \frac{f(x)-f(x_0)}{x-x_0} \right|
= \limsup_{x\to x_0} \frac{|f(x)-f(x_0)|}{|x-x_0|^1}=\infty.
$$
Thus $f$ is not differentiable at $x_0$.
\end{proof}

\begin{lemma}\label{lem:sum}
If $\alpha_1, \alpha_2 ,\ldots, \alpha_n \in (0,1]$ are distinct 
and $f_j$ is \adj $C^{0,\alpha_j}$ at $x_0$, then the
sum $f=f_1 + f_2 + \cdots + f_n$ is \adj $C^{0,\alpha}$ at $x_0$
where $\alpha = \min_{1\le j \le n} 
\alpha_j$.
\end{lemma}

\begin{proof}
  By induction it is enough to show  Lemma~\ref{lem:sum} for $n=2$.
  Without loss of generality, assume $\alpha_1< \alpha_2$.  
Then
\begin{equation}\label{eq:upper}
\frac{|f(x) - f(x_0)|}{|x-x_0|^{\alpha_1}} 
\le \frac{|f_1(x) - f_1(x_0)|}{|x-x_0|^{\alpha_1}} 
    +  \frac{|f_2(x) - f_2(x_0)|}{|x-x_0|^{\alpha_2}} |x-x_0|^{\alpha_2 - \alpha_1}.
\end{equation}
The function $f_2$ is \adj $C^{0,\alpha_2}$ at $x_0$ and therefore
$$
\limsup_{x\to x_0} \frac{|f_2(x) - f_2(x_0)|}{|x-x_0|^{\alpha_2}} < \infty.
$$
Also $\lim_{x\to x_0} |x-x_0|^{\alpha_2 - \alpha_1} =0$ as $\alpha_1<\alpha_2$.
Thus
\begin{equation}\label{eq:lim=0}
    \lim_{x\to x_0} \frac{|f_2(x) - f_2(x_0)|}{|x-x_0|^{\alpha_2}} |x-x_0|^{\alpha_2 - \alpha_1}=0.
\end{equation}
Using this and that $f_1$ is \adj $C^{0,\alpha_1}$ at $x_0$,
inequality \eqref{eq:upper} gives
\begin{align*}
    \limsup_{x\to x_0} \frac{|f(x) - f(x_0)|}{|x-x_0|^{\alpha_1}} 
    &\le\limsup_{x\to x_0}  \frac{|f_1(x) - f_1(x_0)|}{|x-x_0|^{\alpha_1}}  + 0
    <  \infty .
\end{align*}

Let $\beta\in (\alpha_1,\alpha_2)$. 
By the reverse triangle inequality
\begin{equation}
  \label{eq:sumbelow}
\frac{|f(x) - f(x_0)|}{|x-x_0|^{\beta}} \ge
 \frac{|f_1(x) - f_1(x_0)|}{|x-x_0|^{\beta}} 
    -  \frac{|f_2(x) - f_2(x_0)|}{|x-x_0|^{\alpha_2}} |x-x_0|^{\alpha_2-\beta}.
\end{equation}
Since $\beta<\alpha_2$, the argument  showing \eqref{eq:lim=0} holds  gives
$$
\lim_{x\to x_0} \frac{|f_2(x) - f_2(x_0)|}{|x-x_0|^{\alpha_2}} |x-x_0|^{\alpha_2-\beta}=0.
$$
Thus since $\alpha_1<\beta$ and  $f_1$ is \adj $C^{0,\alpha_1}$ at $x_0$,
by  \eqref{eq:sumbelow} 
$$
\limsup_{x\to x_0} \frac{|f(x) - f(x_0)|}{|x-x_0|^{\beta}} 
 \ge \limsup_{x\to x_0}  \frac{|f_1(x) - f_1(x_0)|}{|x-x_0|^{\beta}} - 0 = \infty.
 $$
 Note for  any $\e>0$
 $$
 \limsup_{x\to x_0} \frac{|f(x) - f(x_0)|}{|x-x_0|^{\beta + \e}} \geq
 \limsup_{x\to x_0} \frac{|f(x) - f(x_0)|}{|x-x_0|^{\beta}} = \infty.
 $$
Thus $f$ is \adj $C^{0,\alpha_1}$ at $x_0$.
\end{proof}

\begin{lemma}\label{lem:fou}
  Let $u\cn J \to I$ be a $C^1$ function 
  for intervals  $I$ and $J$ of~$\R$.
Let $f\cn I \to \R$ is \adj $C^{0,\alpha}$ at $u(x_0)$
where $x_0$ is in the interior of~$J$.
\begin{enumerate}
\item[(a)] 
  If $u'(x_0)\ne 0$ then  $f\circ u$
  is \adj $C^{0,\alpha}$ at $x_0$.
\end{enumerate}
Furthermore let the curve  $u$ be $C^{1,\gamma}$ for some $\gamma\in(0,1]$
and   $\alpha > \frac{1}{1+\gamma}$.
\begin{enumerate}
\item[(b)]
  If $u'(x_0) =0$ then $f\circ u$
  is differentiable at $x_0$ and $(f\circ u)'(x_0) =0$.
\end{enumerate}
\end{lemma}

\begin{proof}
    For (a) let $u'(x_0)\ne 0$.  Then, as $u'$ is continuous,
    there is an interval $J_0 = (x_0-\delta , x_0+\delta)$
    such that  $u'$ does not change sign on~$J_0$ and,
    for some constant $C>0$,
    we have $C^{-1} \le |u'(x)| \le C$ for each $x\in J_0$.
    So by the mean value theorem
    $$
    \frac1C |x-\widetilde{x}| \le |u(x)-u(\widetilde{x})|
     \le C |x-\widetilde{x}|
    $$
     holds for $x,\widetilde{x}\in J_0$. 
     Let $I_0 =\{ u(x): x\in J_0\}$. 
     The last inequality gives that
     $u$ is a Lipschitz bijection from $J_0$ onto $I_0$
     with a Lipschitz inverse.       
Hence, since 
$f$ is \adj  $C^{0,\alpha}$ at $u(x_0)$,
\begin{align*}
\limsup_{x\to x_0}&\frac{|f(u(x))- f(u(x_0))|}{|x-x_0|^\alpha}
\\
&=
\limsup_{x\to x_0} \frac{|f(u(x))- f(u(x_0))|}{|u(x)-u(x_0)|^\alpha} 
\left| \frac{u(x)-u(x_0)}{x-x_0}    \right|^\alpha
\\
&\le C^\alpha \limsup_{x\to x_0} \frac{|f(u(x))- f(u(x_0))|}{|u(x)-u(x_0)|^\alpha}
\\
&= C^\alpha \limsup_{y\to u(x_0)} \frac{|f(y)- f(u(x_0))|}{|y-u(x_0)|^\alpha}
\\ 
&< \infty 
\end{align*}
where we have done the change of variable $y=u(x)$ in the last $\limsup$.

Similarly, for each $\beta>\alpha$ 
\begin{align*}
        \limsup_{x\to x_0} &\frac{|f(u(x))- f(u(x_0))|}{|x-x_0|^\beta} \\
          &=\limsup_{x\to x_0} \frac{|f(u(x))- f(u(x_0))|}{|u(x)-u(x_0)|^\beta} 
             \left| \frac{u(x)-u(x_0)}{x-x_0}    \right|^\beta\\
          &\ge \frac1{C^\beta} \limsup_{x\to x_0} \frac{|f(u(x))- f(u(x_0))|}{|u(x)-u(x_0)|^\beta} \\
          &= \frac1{C^\beta }\limsup_{y\to u(x_0)} \frac{|f(y)- f(u(x_0))|}{|y-u(x_0)|^\beta}\\
&=\infty.
\end{align*}
Thus $f\circ u$ is \adj $C^{0,\alpha}$ at $x_0$.

For (b) we have $u'(x_0)=0$ and 
constants $\gamma\in(0,1]$ and  $\rho>0$ 
so that  $|u'(x) - u'(y)| \le \rho |x-y|^\gamma$ for  $x,y\in J$.
By the mean value theorem, for $x$  in the interior of $J$
there is a $\xi$ between $x$ and $x_0$ with
\allowdisplaybreaks
\begin{align*}
    |u(x)-u(x_0)|&= |u'(\xi)| |x-x_0|\\
                 &= |u'(\xi)-u'(x_0)||x-x_0|&& \text{(as $u'(x_0)=0$)}\\
                 &\le \rho |\xi - x_0|^\gamma \, |x-x_0|\\
    & \le \rho |x-x_0|^\gamma \, |x-x_0|&& \text{(as $\xi$ between $x$ and $x_0$)}\\
                & = \rho|x-x_0|^{1+\gamma}.
\end{align*}
Since $f$ is  \adj  $C^{0,\alpha}$ at $u(x_0)$, there is a constant
$K>0$ such that
$|f(y) - f(u(x_0))| \le K|y-u(x_0)|^\alpha$
for $y$ in some neighborhood of $u(x_0)$.
Thus for $x$ close to $x_0$ 
$$
|f(u(x)) - f(u(x_0))|\le K|u(x)- u(x_0)|^\alpha  \le
K\rho^\alpha |x-x_0|^{\alpha(1+\gamma)}
$$
and $\alpha(1+\gamma) >1$ by the  lower bound on  $\alpha$
thus, $(f\circ u)'(x_0)=0$.
\end{proof}

\begin{thm}\label{thm:main_example}
  Let $\gamma\in(0,1]$ and 
  $\alpha_1, \alpha_2 ,\ldots, \alpha_d$ be distinct 
  numbers in the interval $(\tfrac{1}{1+\gamma},1)$.
  For each
  $j \in \{1,2, \ldots, d\}$, let the function 
  $f_j\cn\R\to\R$ be  \adj $C^{0,\alpha_j}$
  at each point in $\R$. 
Define  $f\cn \R^d \to \R$ by 
$$
f(x_1, x_2 ,\ldots,x_d) = \sum_{j=1}^d f_j(x_j). 
$$
Then $f$ is continuous and, 
for any $C^{1,\gamma}$ test curve $c: [a,b]\to\R^d$,
the composite $f\circ c$ is nowhere differentiable on $(a,b)$.

Moreover,
for any $C^{1,\gamma}$ test curve $c\cn [a,b]\to \R^d$
and any $s_0\in (a,b)$, there
is an $\alpha \in \{ \alpha_1, \alpha_2 ,\ldots, \alpha_d\}$ such
that $f\circ c $ is \adj $C^{0,\alpha}$ at $s_0$. 
\end{thm}

A concrete choice of the function $f_j$ is to let $f_j$
be the \adj $C^{0,\alpha}$ function
constructed in Theorem~\ref{thm:strict_fcn} 
with $\alpha=\alpha_j$.

\begin{proof}
  Note $f$ is continuous since it is the sum of  continuous function.
  Let $\gamma\in(0,1]$.
    Let $c \cn [a,b]\to \R^d$ be a $C^{1,\gamma}$ test curve. 
   Fix $s_0\in (a,b)$.
   By Proposition~\ref{prop:not_diff},
    it  only remains  to show 
    $f\circ c$ is \adj $C^{0,\alpha_j}$ at~$s_0$  for some~$j$.

    Write the test curve $c$ as 
    $$
    c(s) = ( u_1(s), u_2(s) ,\ldots,u_d(s)).
    $$
    Then each function $u_j$ is $C^1$.  In this notation
    $$
    (f\circ c) (s) = \sum_{j=1}^d (f_j\circ u_j)(s).
    $$
    As $c\in C^{1,\gamma}$, there is a 
    $\rho>0$ so that $\| c'(s_2) - c'(s_1)\| \le \rho |s_2 - s_1|^\gamma$;
    thus, $|u_j'(s_2) - u_j'(s_1)| \le \rho |s_2- s_1|^\gamma$ 
    for all~$j$ and $s_1, s_2\in [a,b]$. 
    Decompose $f\circ c$ as $f\circ  c =f_A + f_B$ where
    \begin{equation*}
        A = \{ j : u_j'(s_0) \ne 0\}, \qquad 
        B= \{ j : u_j'(s_0) = 0\}
    \end{equation*}
    and, following the convention that the sum over the empty set is zero, 
    \begin{equation*}
    f_A = \sum_{j\in A} f_j\circ u_j, \qquad f_B = \sum_{j\in B} f_j\circ u_j.
    \end{equation*}
    By Lemma \ref{lem:fou},   
    $f_j \circ u_j$ is \adj $C^{0,\alpha_j}$ at $s_0$ when $j\in A$
    while $f_j\circ u_j$ has $(f_j\circ u_j)'(s_0)=0$ when $j \in B$.
    As $c'(s_0)\ne 0$, it follows $A\ne \varnothing$;
    let $\alpha = \min_{j\in A}\alpha_j$.
    By Lemma \ref{lem:sum}, 
    $f_A$ \adj  $C^{0,\alpha}$ at $s_0$. 
    As $f_B$ is differentiable at $s_0$,
    the sum $f\circ c=f_A + f_B$ is \adj $C^{0,\alpha}$ at $s_0$.
\end{proof}

\section{Proof of the main theorem}\label{section:PfMainThm} 

Let $U$ be a bounded open set in $\R^d$. Let $\overline U$ be
the closure of $U$ and $\f U$ be the boundary of $U$.
For a nonnegative integer $n$,  let
\begin{equation*}
K_n := \{ u\in U: \| u-x\| \ge 1/n \text{ for each  $x\in \f U$}\}
\end{equation*}
be the set of
points in $U$ that are  a distance of at least $1/n$
from $\f U$.

\begin{defn}\label{def:Cn}
Let  $n$ be a positive integer and $\gamma\in (0,1]$.
Let $\mathcal C_n^\gamma (U)$ be the 
set of test curves  $c\cn [-1/n,1/n] \to U$ satisfying 
\begin{enumerate}[(a)]
    \item $\|c'(s_2) - c'(s_1)\| \le n | s_2 - s_1|^{\gamma} $ for all 
      $s_1, s_2\in [-1/n,1/n]$
    \item $\| c(s) - x\|\ge 1/n$ for all $s\in [-1/n,1/n]$ and $x\in \f U$.  
\end{enumerate}
\end{defn}

\begin{lemma}\label{lem:C_nc_cpt}
  The set $\mathcal C_n^\gamma (U)$ is compact in the $C^1$ topology.
\end{lemma}

Recall the $C^1$ topology on $C^1$ functions $c\cn [a,b]\to U$
is given by the
norm $\|c\|_{C^1} = \|c\|_{L^\infty} + \|c'\|_{L^\infty}$.

\begin{proof}
Note $K_n$ is a closed bounded, and thus
compact, subset of $\R^d$.
Condition (b) in  Definition \ref{def:Cn} 
can be restated as $c \cn [-1/n,1/n] \to K_n$.
Let $\la c_\ell \ra_{\ell=1}^\infty$
be a sequence from $\mathcal C_n^\gamma (U)$.
It suffices to find a subsequence of $\la c_\ell \ra_{\ell=1}^\infty$
that converges in the $C^1$ topology
to an element of $\mathcal C_n^\gamma (U)$.

The set $\{ c' \cn c\in \mathcal C_n^\gamma (U)\}$ 
is uniformly equicontinuous as the functions $c'$ satisfy
the uniform H\"older condition in (a) and this set is bounded in the
sup norm topology 
as the curves
are unit speed.  
Therefore the Arzel\`a–Ascoli Theorem
(cf.\ \cite[p.~137]{Folland:Analysis}) implies 
the closure of  $\{ c' \cn c\in \mathcal C_n^\gamma (U)\}$ is compact
sup norm topology. 
So we can pass to a subsequence and
assume $\la c'_\ell \ra_{\ell=1}^\infty$
converges uniformly on $[-1/n,1/n]$ to some function.
By the compactness of $K_n$ we can pass to
a further subsequence and assume 
$\la c_\ell (0) \ra_{\ell=1}^\infty$
converges to some point in $K_n$. 
 Thus (cf.~\cite[\S8.2]{BartleSherbert4})
 there is a differentiable function $c\cn [-1/n,1/n] \to \R^d$
 such that
\begin{align*}
 \lim_{\ell \to \infty} c_\ell &= c \\
 \lim_{\ell \to \infty} c_\ell' &= c'
\end{align*}
with both convergences being uniformly on  $[-1/n,1/n]$.
So $c$ has unit speed as each $c_\ell$ does.
Condition (a) holds since
if $s_1, s_2\in [-1/n,1/n]$
then 
 $
 \| c'(s_2) - c'(s_1)\| = \lim_{\ell \to \infty}
 \| c_\ell'(s_2) - c_\ell'(s_1) \| \le n |s_2-s_1|^{\gamma}
 $.
The curve~$c$ is $K_n$-valued as each $c_\ell$ is and $K_n$ is closed.
Thus  $c\in \mathcal C_n^\gamma (U)$.
\end{proof}

\begin{defn}
 Let  $n$ be a positive integer and $\gamma\in (0,1]$.
Let $\mathcal F_n^\gamma (U)$ be the set of $f\in C(\overline U)$ such that 
for some $c \in \mathcal C_n^\gamma (U)$  the  inequality
\begin{equation}\label{eq:F_d-def}
| f(c(s)) - f(c(0)) |\le n |s|
\end{equation}
holds for all $s\in [-1/n,1/n]$.
\end{defn}

\begin{lemma}\label{lem:not_dense}
  The set $\mathcal F_n^\gamma (U)$ is a closed nowhere dense
  subset of $C(\overline U)$.
\end{lemma}

\begin{proof}
Let $\gamma \in (0,1]$. 
Let $\la f_\ell\ra_{\ell =1}^\infty$ be a sequence in $\mathcal F_n^\gamma (U)$
with  $f_\ell \to f$ uniformly  on~$\overline U$
for some   $f\in C(\overline U)$.
To show $\mathcal F_n^\gamma (U)$ is closed it is enough to show
$f\in\mathcal F_n^\gamma (U)$. 
For each $\ell$ there is
 $c_\ell \in \mathcal C_n^\gamma (U)$
with 
$$
| f_\ell(c_\ell(s)) - f_\ell(c_\ell(0))|\le n |s|
$$
for all $s\in [-1/n,1/n]$.
Since  $\mathcal C_n^\gamma (U)$ is compact, 
passing to a subsequence, we  assume there
is a $c\in \mathcal C_n^\gamma (U)$  such that
$c_\ell \to c$ in the $C^1$ topology.
Then
$$
\lim_{\ell \to \infty} f_\ell(c_\ell(s)) = f(c(s))
$$
for each  $s\in [-1/n,1/n]$ since
$f$ is continuous, $c_\ell(s)\to c(s)$, and 
\begin{align*}
  |f_\ell(c_\ell(s))- f(c(s))|&\le |f_\ell(c_\ell(s)) - f(c_\ell(s))| 
        + |f(c_\ell(s)) - f(c(s))|\\
    &\le \| f_\ell - f\|_{L^\infty} + |f(c_\ell(s)) - f(c(s))|\\
    & \stackrel{\ell\to \infty}{\longrightarrow} 0.
\end{align*}
Therefore
$$
|f(c(s)) - f(c(0))| = \lim_{\ell\to \infty} |f_\ell(c_\ell(s))-f_\ell(c_\ell(0))|
\le n |s|
$$
for all $s\in [-1/n,1/n]$.  
Thus $f\in \mathcal F_n^\gamma (U)$.
So  $\mathcal F_n^\gamma (U)$ is closed.

It remains to show $\mathcal F_n^\gamma (U)$ is nowhere dense. 
Toward a contradiction assume this is not the case.
Then the closed set $\mathcal F_n^\gamma (U)$ has nonempty interior. 
Thus $\mathcal F_n^\gamma (U)$ contains 
a nonempty  open subset $\mathcal V$ of $C(\overline U)$.
As  the  smooth functions are dense in $C(\overline U)$
there is a $f_0\in \mathcal V$ that is~$C^\infty$.
Let $f$ be the function constructed in  Theorem~\ref{thm:main_example}.

For a small enough positive $\delta$
the sum $f_\delta:=f_0+\delta f$ 
is in the open set~$\mathcal V$ and
so $f_\delta \in  \mathcal F_n^\gamma (U)$.
By the definition of $ \mathcal F_n^\gamma (U)$
there is~$c\in \mathcal C_n^\gamma  (U)$ so that
the Lipschitz inequality
\begin{equation*}
    |f_\delta(c(s)) - f_\delta(c(0))| \le n |s|
\end{equation*}
holds for each  $s\in[-1/n,1/n]$.
However, by Theorem~\ref{thm:main_example}, $f\circ c$ is \adj~$C^{0,\alpha}$
at $0$ for some $\alpha\in (0,1)$
since the curve $c$ is $C^{1,\gamma}$.
As~$f_0$ is smooth, $f_0\circ c$ is~$C^1$.
Thus the sum $f_\delta\circ c = f_0\circ c +  \delta ( f\circ c) $
is \adj~$C^{0,\alpha}$ at~$0$ where $\alpha<1$.
Hence 
\begin{equation*} 
   \limsup_{s\to 0}  \frac{|f_\delta(c(s)) -f_\delta(c((0))|}{|s-0|} = \infty, 
  \end{equation*} 
contradicting  the previous 
inequality.
So  $\mathcal F_n^\gamma (U)$ is nowhere dense.
\end{proof}

\begin{proof}[Proof of the Main Theorem]
Let 
\begin{equation*}
\mathcal G :=  C(\overline U) \setminus 
 \bigg(\bigcup_{n=1}^\infty \mathcal F_n^{1/n} (U)\bigg).
\end{equation*}
Then $\mathcal G$ is a  $G_\delta$ 
set that is dense in  $C(\overline U)$ by
the Baire category theorem and Lemma~\ref{lem:not_dense}.
It suffices to   show  the functions in $\mathcal G$
are strongly nowhere differentiable on $U$.
Let $f\in \mathcal G$.

Let $c\cn [a,b]\to U$ be a test curve.
Towards a  contradiction,
assume $f\circ c$ is  differentiable at some  $s_0\in (a,b)$.
Then there
is a constant $M$ such that for each $s\in [a,b]$
\begin{equation}\label{eq:f(c)}
| (f \circ c) (s) - (f\circ c)(s_0)| \le M |s-s_0|.
\end{equation}
As $c$ is a test curve,
there is a  $\gamma\in (0,1]$ and a constant $\rho$ so that
$\| c'(s_2) - c'(s_1)\| \le \rho |s_2- s_1|^\gamma$
for $s_1, s_2 \in [a,b]$.
Choose $n$ so that $n\geq\max(M,\tfrac{1}{\gamma},\rho,2)$
and   $[s_0-1/n,s_0+1/n]\subseteq (a,b)$.
The image of $c$, being compact, is a positive distance
from the boundary of the open set~$U$;
thus, by possibly increasing the size of $n$, we can also assume
$c(s)\in K_n$ for each  $s\in [a,b]$.

Define $\tilde c \cn [-1/n,1/n] \to U$ by 
$$
\tilde c(t) = c(t + s_0). 
$$
Then $\tilde c$ is of unit speed
and $K_n$-valued
as
$c$ is.
For  $t_1, t_2\in [-1/n,1/n]$,
we have  $| t_2-t_1 | \leq 1$ as $n\geq 2$ and so,
since $\rho \leq n$ and $1/n \leq \gamma$,
\begin{align*}
  \| \tilde c'(t_2) - \tilde c'(t_1)\|
   & = \| c'(t_2 + s_0)   - c'(t_1 + s_0 )\|\\
      &\le \rho |t_2-t_1|^\gamma \le n |t_2-t_1|^{1/n}.
\end{align*} 
Thus $\tilde c \in \mathcal C_n^{1/n} (U)$.
The inequality in \eqref{eq:f(c)}
along with $M\le n$ give,
for each $t\in [-1/n,1/n]$,
\begin{align*}
|f(\tilde c(t)) - f(\tilde c(0))|
 = |f(c(t+s_0)) - f(c(s_0))| 
 \le M|t| \le n|t| .
\end{align*} 
Hence  $f\in \mathcal F_n^{1/n} (U)$, which contradicts  $f \in \mathcal G$.
So $f\circ c$ is  not differentiable at $s_0$.
Thus $f\circ c$ is nowhere differentiable on $(a,b)$. 

Hence  $f$ is strongly nowhere differentiable, completing the proof. 
\end{proof}

\section{Open questions}

In the Main Theorem we test multivariant functions
for differentiability by testing along each unit speed curve $C^1$
which is also
$C^{1,\gamma}$ for some $\gamma\in (0,1]$.
In terms of checking for  differentiability,
the extra smoothness assumption is not particularly natural. 
The following questions,
which we state as conjectures, seem natural.

\begin{conjecture}\label{conj:1}
There is a continuous function $f\cn \R^d \to \R$ whose restriction
to any  $C^1$ unit speed curve is nowhere differentiable. 
\end{conjecture}

\begin{conjecture}\label{conj:2}
Let $U$ is a bounded open set in $\R^d$.
The set of  $C(\overline U)$ functions 
whose restriction to each  $C^1$ unit speed curve
is nowhere differentiable 
contains a dense $G_\delta$ subset.
\end{conjecture}

If
the first  conjecture holds, it would be nice if the second one
followed by  arguments similar to the ones used to prove our Main Theorem.
But our proof uses in an essential way the compactness result of Lemma
\ref{lem:C_nc_cpt}. In our case the compactness follows from 
the  Arzel\`a–Ascoli Theorem and the assumption the derivatives
of the test curves satisfy a H\"older condition.  If we are using
as the set of test curves all $C^1$ unit speed curves, a direct adaptation 
of our methods would require  that the set of all unit speed
curves $c\cn [-1/n,1/n]\to U$ is a countable union of compact sets
in the $C^1$ topology, which is false.
Thus, even if Conjecture \ref{conj:1} holds,
new ideas are  needed to prove Conjecture~\ref{conj:2}.

\section{Appendix: Katzourakis' example}

The
first person to show a function having a property similar to
being \adj  $C^{0,\alpha}$ was Hardy in \cite[Theorem~1.32]{Hardy:Reim_fcn}
where he shows the Weierstrass function
$W(x) = \sum_{n=1}^\infty a^n \cos(b^n\pi x)$
with $0<a<1$, $ab>1$, and $\alpha = \ln(1/a)/\ln(b)$ 
satisfies the pointwise estimates
$$
0 < \limsup_{y\to x} \frac{|W(x) - W(y)|}{|x-y|^\alpha} < \infty.
$$
However his proof does not seem to show this inequality
holds uniformly in~$x$ and thus   does not show
that $W$ is $\alpha$-H\"older on $\R$.  
Hardy's proof proceeds by extending~$W$ to a  harmonic function 
on the upper half plane $\{ x+iy: y\ge 0\}$  and   using methods
and results from harmonic analysis.

The first examples of  \adj $C^{0,\alpha} $ functions were
given by E.\ I.\ Berezhno\u{i} \cite{Berez:nonsmooth}. 
He uses a modulus of continuity $\omega$ more general than
the H\"older case $\omega(t) = t^\alpha$ and  uses rather
abstract methods from Banach space theory.
In  \cite{Kat:Holder},   
Nikolaos I.\  Katzourakis constructs  concrete  examples
of functions $K_{\alpha, \nu}$ that are \adj $C^{0,\alpha}$. 
We give a variant of his construction.
We get a larger class of functions
(we only require that $b$ in the definition \eqref{eq:Phi_def}
is an even integer with $b^{1-\alpha}>2$ while 
Katzourakis requires  $b=2^{2\nu}$ 
for  a sufficiently large integer $\nu$). 
Our elementary Lemma \ref{lem:a<x<b} 
lets us simplify the proof a bit
by only having to look at difference quotients on intervals of
the form $[b^{-m}j,b^{-m}(j+1)]$.
However the main reason our construction  
is shorter is that  Katzourakis proves his functions also
satisfy the stronger   ``Hardy condition''
\begin{equation*}
0< \limsup_{y\to x} \frac{|K_{\alpha,\nu}(y)- K_{\alpha,\nu}(x)|}{|y-x|^\alpha}
 < \infty
\end{equation*}
for each $x\in \R$.
We only show the less precise lower bound of~\eqref{eq:b-infty}.
 Proving Katzourakis' more precise result 
requires a substantial amount of extra work.  We include this example 
for completeness and because it  requires less machinery 
than any example we have seen in the literature.

Define a sawtooth function  $\phi\cn \R \to \R$ by
$$
\phi(x) := \dist(x, 2\Z).
$$
Then $\phi$ is continuous,   even, 2-periodic, and  satisfies
for each $x,y\in\R$
$$
0\le \phi(x) \le 1
$$
and
\begin{equation}\label{eq:phi_lip}
|\phi(x)-\phi(y)|\le |x-y|.
\end{equation}
For $0<\alpha < 1$
and an   even positive integer $b$,
define 
\begin{equation}\label{eq:Phi_def}
\Phi(x) := \sum_{k=0}^\infty b^{-k\alpha}\phi(b^k x).
\end{equation}
Then $\Phi$ is continuous,  even, 2-periodic, and thus bounded.

\begin{thm}\label{thm:Kat}
  The function $\Phi$ is \adj $C^{0,\alpha}$ on $\R$
  when  $0<\alpha<1$, the number $b$ is an even positive integer,
  and  $b^{1-\alpha}>2$. 
\end{thm}

Theorem \ref{thm:Kat}  follows from
Proposition \ref{prop:C-alpha} and Proposition \ref{prop:sup=infty}.

\begin{prop}\label{prop:C-alpha} 
  The function $\Phi$ is in $C^{0,\alpha}(\R)$ when $0<\alpha<1$
  and the  $b$ is an even positive integer.
\end{prop}

\begin{proof}
  Since $\Phi$ is even and 2-periodic,
   it is enough to show $\Phi$ is  
   H\"older continuous with exponent $\alpha$
   on the interval $[0,1]$.  Let $x,y\in [0,1]$.
    Without loss of generality we may assume $x\ne y$.  
    Let $m$ be a positive integer with
    $$
    b^{-m} < |x-y| \le b^{-m+1}.
    $$
    As $\phi$ takes its values in $[0,1]$  we have
    $|\phi(b^kx)-\phi(b^ky)|\le 1$. Using this,
    the inequality \eqref{eq:phi_lip}, and the triangle inequality
    \begin{align}\nonumber
        | \Phi(x) - \Phi(y)|&\le \sum_{k=0}^{m} b^{-k\alpha}|\phi(b^kx) - \phi(b^ky)|\\
            &\qquad\quad  +\sum_{k=m+1}^\infty b^{-k\alpha}|\phi(b^kx) - \phi(b^ky)|\nonumber\\
            &\le \sum_{k=-\infty}^m b^{k(1-\alpha)}|x-y| + \sum_{k=m+1}^\infty b^{-k\alpha}\nonumber\\
        \label{eq:geo}           &= \frac{b^{m(1-\alpha)}}{1-b^{-(1-\alpha)}} |x-y| + \frac{b^{-(m+1)\alpha}}{1-b^{-\alpha}}\\
                                 &= \frac{b^{m(1-\alpha)}|x-y|^{1-\alpha}}{1-b^{-(1-\alpha)}} |x-y|^\alpha + \frac{b^{-m\alpha}b^{-\alpha}}{1-b^{-\alpha} }\nonumber
\end{align}
where the equality \eqref{eq:geo} comes from summing the two geometric series.
As $|x-y|\le b^{-m+1}$ 
$$
b^{m(1-\alpha)}|x-y|^{1-\alpha}\le b^{m(1-\alpha)}( b^{-m+1})^{1-\alpha}
= b^{1-\alpha}.
$$
The inequality $b^{-m}\le |x-y|$ implies
$$
b^{-m\alpha}\le |x-y|^\alpha.
$$
Using these inequalities in the estimate for $|\Phi(x)-\Phi(y)|$
gives
\begin{align*}
    |\Phi(x)-\Phi(y)|&\le  \frac{b^{1-\alpha}}{1-b^{-(1-\alpha)}} |x-y|^\alpha + \frac{b^{-\alpha}}{1-b^{-\alpha} }|x-y|^\alpha\\
                      &= C_\alpha |x-y|^\alpha
\end{align*}
where 
$$
C_\alpha =  \frac{b^{1-\alpha}}{1-b^{-(1-\alpha)}} +  \frac{b^{-\alpha}}{1-b^{-\alpha} }~,
$$
which shows $\Phi$ is H\"older continuous on with exponent $\alpha$ on $[0,1]$.
\end{proof}

\begin{lemma}\label{lem:a<x<b}
    Let $f\cn \R \to \R$ and let $x\in \R$.  Assume for some
    $\beta\in (0,1]$ there is a constant $C$ with
    $$
    |f(x)-f(y)| \le C |x-y|^\beta
    $$
    for  all $y\in \R$.  Let $x_0\le x \le x_1$.  Then
    $$
    |f(x_1)-f(x_0)| \le 2C |x_1-x_0|^\beta.
    $$
\end{lemma}

\begin{proof}
    \begin{align*}
    |f(x_1) - f(x_0)|&\le |f(x_1)-f(x)|+ |f(x)-f(x_0)|\\ 
                 &\le C|x_1-x|^\beta + C |x-x_0|^\beta \le 2C|x_1-x_0|^\beta.
    \end{align*}
\end{proof}

\begin{prop}\label{prop:sup=infty}
  Let $\Phi$ be as in \eqref{eq:Phi_def} so  $0<\alpha<1$ and the 
  $b$ is an even positive integer.  
  Assume furthermore that  $b^{1-\alpha}>2$. 
  Let  $\beta > \alpha$.   Then for each $x\in\R$ 
$$
\limsup_{y\to x} \frac{|\Phi(x)-\Phi(y)|}{|x-y|^\beta} = \infty.
$$
\end{prop}

\begin{proof}
  Let $m$ be a positive integer.
  Let $j$ and $k$ be nonnegative integers.
  As $\phi$ is periodic with period $2$
  when  $k>m$ the power $b^{k-m}$ is an even integer (as $b$ is even)
  and so
$$
\phi(b^k(b^{-m}(j+1))) - \phi(b^k(b^{-m}j))
= \phi(b^{k-m}j + b^{k-m}) - \phi(b^{k-m}j) =0.
$$
Therefore the sum for $\Phi(b^{-m}(j+1))- \Phi(b^{-m}j)$ reduces to a finite
sum
\begin{align*}
  \Phi(b^{-m}(j+1))&- \Phi(b^{-m}j)
  \\
   &=
  \sum_{k=0}^m b^{-k\alpha}\left[\phi(b^k(b^{-m}(j+1))) - \phi(b^k(b^{-m}j))\right]\\
     &= b^{-m\alpha}\left[ \phi(j+1) - \phi(j)\right]\\
  &\qquad\qquad+ \sum_{k=0}^{m-1} b^{-k\alpha}
     \left[  \phi(b^{k-m}j + b^{k-m}) - \phi(b^{k-m}j)    \right]. 
\end{align*}
But $\phi(j+1) - \phi(j) = \pm 1$ for all $j$.
As  $\phi$ is Lipschitz with
Lipschitz constant one, 
$|\phi(b^{k-m}j + b^{k-m}) - \phi(b^{k-m}j)   |  \le  b^{k-m}$.
Putting this together with reverse triangle inequalitygives

\begingroup
\allowdisplaybreaks
\begin{align*}
    |\Phi(b^{-m}(j+1))&- \Phi(b^{-m}j)|\\
 &\ge  b^{-\alpha m} - \sum_{k=0}^{m-1} b^{-\alpha k} | \phi(b^{k-m}j + b^{k-m}) - \phi(b^{k-m}j)  |\\
 &\ge b^{-\alpha m} - \sum_{k=0}^{m-1} b^{-k\alpha} b^{k-m}\\
&= b^{-\alpha m} - b^{-m} \sum_{k=0}^{m-1} b^{k(1-\alpha)} \\
&= b^{-\alpha m} - b^{-m}                    
           \frac{1 - b^{m(1-\alpha)}}{1 - b^{1-\alpha}} \\
&= b^{-\alpha m} \left( 1 - b^{-m + \alpha m}\,  
                \frac{1 - b^{m(1-\alpha)}}{1 - b^{1-\alpha}}\right) \\
&= b^{-\alpha m} \left( 1 - b^{-m(1-\alpha)}\,  
                \frac{ b^{m(1-\alpha)}-1 }{ b^{1-\alpha}-1}\right) \\
&= b^{-\alpha m} \left( 1 - \frac{1 - b^{-m(1-\alpha)} }{ b^{1-\alpha}-1}\right)\\
&\ge b^{-\alpha m} \left( 1 - \frac{1  }{ b^{1-\alpha}-1}\right). 
\end{align*}%
\endgroup
Thus
\begin{equation}\label{eq:Phi-diff-lower}
    |\Phi(b^{-m}(j+1))- \Phi(b^{-m}j)| \ge M_\alpha b^{-\alpha m} 
\end{equation}
where
$$
M_\alpha =  1 - \frac{1  }{ b^{1-\alpha}-1}.
$$
When $b^{1-\alpha}>2$ 
$$
    M_\alpha >0.
$$

Towards a contradiction assume there is a $x\in \R$ with
$$
\limsup_{y\to x} \frac{|\Phi(y) - \Phi(x)|}{|y-x|^\beta} <\infty.
$$
By the periodicity of $\Phi$ we can assume $x \geq 0$.
As $\Phi$ is bounded
this implies there is a $C>0$ so that
$$
|\Phi(y)-\Phi(x)|\le C |x-y|^\beta
$$
holds for all $y\in \R$.
For each positive integer $m$ choose a nonnegative integer  $j$ such
that 
$$
b^{-m} j \le x < b^{-m}(j+1).
$$
Letting $x_0=b^{-m} j $ and $x_1 = b^{-m}(j+1)$ in Lemma \ref{lem:a<x<b} 
$$
|\Phi(x_1) - \Phi(x_0)| \le 2 C |x_1- x_0|^\beta = 2C b^{-\beta m}.
$$
Combining this with the inequality \eqref{eq:Phi-diff-lower} gives 
$$
M_\alpha b^{-\alpha m}\le  |\Phi(x_1) - \Phi(x_0)| \le 2C b^{-\beta m}, 
$$
which implies
$$
M_\alpha b^{m(\beta - \alpha)} \le 2C.
$$
Using $M_\alpha >0$ and  $\beta-\alpha>0$, letting $m\to \infty$ gives
the required  contradiction.
\end{proof}

\end{document}